\documentclass[12pt]{amsart}

\usepackage{amssymb}

\usepackage{enumerate}

\usepackage{cite}

\usepackage{hyperref}

\hypersetup{%
pdftitle={},
pdfsubject={Mathematics},
pdfauthor={Manfred Madritsch},
pdfkeywords={}
hyperindex=true,plainpages=false}

\makeatletter
\@namedef{subjclassname@2010}{%
  \textup{2010} Mathematics Subject Classification}
\makeatother

\newtheorem{thm}{Theorem}[section]

\newtheorem{lem}[thm]{Lemma}
\newtheorem{prop}[thm]{Proposition}

\theoremstyle{definition}

\newtheorem{rem}[thm]{Remark}

\numberwithin{equation}{section}

\newcommand{\N}{\mathbb{N}}
\newcommand{\Z}{\mathbb{Z}}

\newcommand{\R}{\mathbb{R}}

\newcommand{\lf}{\left\lfloor}
\newcommand{\rf}{\right\rfloor}

\renewcommand{\lvert}{\left\vert}
\renewcommand{\rvert}{\right\vert}

\begin{document}




\title{Construction of normal numbers via pseudo polynomial prime sequences}

\author[M. G. Madritsch]{Manfred G. Madritsch}
\address{
\noindent 1. Universit\'e de Lorraine, Institut Elie Cartan de Lorraine, UMR 7502, Vandoeuvre-l\`es-Nancy, F-54506, France;\newline
\noindent 2. CNRS, Institut Elie Cartan de Lorraine, UMR 7502, Vandoeuvre-l\`es-Nancy, F-54506, France}
\email{manfred.madritsch@univ-lorraine.fr}

\date{\today}

\begin{abstract}
  In the present paper we construct normal numbers in base $q$ by
  concatenating $q$-ary expansions of pseudo polynomials evaluated at
  the primes. This extends a recent result by Tichy and the
  author.
\end{abstract}



\maketitle

\section{Introduction}\label{sec:introduction}
Let $q\geq 2$ be a positive integer. Then every real $\theta\in[0,1)$
admits a unique expansion of the form
\[\theta=\sum_{k\geq1}a_kq^k\quad(a_k\in\{0,\ldots,q-1\})\]
called the $q$-ary expansion. We denote by
$\mathcal{N}(\theta,d_1\cdots d_\ell,N)$ the number of occurrences of the block
$d_1\cdots d_\ell$ amongst the first $N$ digits, \textit{i.e.}
\[\mathcal{N}(\theta,d_1\cdots d_\ell,N):=\#\{0\leq i< n\colon
a_{i+1}=d_1,\ldots,a_{i+\ell}=d_\ell\}.\] Then we call a number normal
of order $\ell$ in base $q$ if for each block of length $\ell$ the
frequency of occurrences tends to $q^{-\ell}$. As a qualitative
measure of the distance of a number from being normal we introduce for
integers $N$ and $\ell$ the discrepancy of $\theta$
by \[\mathcal{R}_{N,\ell}(\theta)=\sup_{d_1\ldots
  d_\ell}\lvert\frac{\mathcal{N}(\theta,d_1\cdots
  d_\ell,N)}{N}-q^{-k}\rvert,\] where the supremum is over all blocks
of length $\ell$. Then a number $\theta$ is normal to base $q$ if for
each $\ell\geq1$ we have that $\mathcal{R}_{N,\ell}(\theta)=o(1)$ for
$N\to\infty$. Furthermore we call a number absolutely normal if it is
normal in all bases $q\geq2$.

Borel \cite{borel1909:les_probabilites_denombrables} used a slightly
different, but equivalent (\textit{cf.} Chapter 4 of \cite{bugeaud2012:distribution_modulo_one}), definition of normality to show that almost
all real numbers are normal with respect to the Lebesgue
measure. Despite their omnipresence it is not known whether numbers
such as $\log2$, $\pi$, $e$ or $\sqrt2$ are normal to any base. The
first construction of a normal number is due to Champernowne
\cite{champernowne1933:construction_decimals_normal} who showed that
the number
\begin{align*}
0.1\,2\,3\,4\,5\,6\,7\,8\,9\,10\,11\,12\,13\,14\,15\,16\,17\,18\,19\,20\dots
\end{align*}
is normal in base $10$.

The construction of Champernowne laid the base for a class of
normal numbers which are of the form
\begin{gather*}
\sigma_q=\sigma_q(f)=
  0.\lf f(1)\rf_q\lf f(2)\rf_q\lf f(3)\rf_q \lf f(4)\rf_q \lf f(5)\rf_q \lf f(6)\rf_q \dots,
\end{gather*}
where $\lf\cdot\rf_q$ denotes the expansion in base $q$ of the integer
part. Davenport and Erd{\H o}s
\cite{davenport_erdoes1952:note_on_normal} showed that $\sigma(f)$ is
normal for $f$ being a polynomial such that $f(\N)\subset\N$. This
construction was extended by Schiffer
\cite{schiffer1986:discrepancy_normal_numbers} to polynomials with
rational coefficients. Furthermore he showed that for these
polynomials the discrepancy $\mathcal{R}_{N,\ell}(\sigma(f))\ll (\log
N)^{-1}$ and that this is best possible. These results where extended
by Nakai and Shiokawa
\cite{nakai_shiokawa1992:discrepancy_estimates_class} to polynomials
having real coefficients. Madritsch, Thuswaldner and Tichy
\cite{madritsch_thuswaldner_tichy2008:normality_numbers_generated}
considered transcendental entire functions of bounded logarithmic
order. Nakai and Shiokawa
\cite{nakai_shiokawa1990:class_normal_numbers} used
pseudo-polynomial functions, \textit{i.e.} these are function of the
form
\begin{gather}\label{mani:pseudopoly}
  f(x)=\alpha_0 x^{\beta_0}+\alpha_1x^{\beta_1}+\cdots+\alpha_dx^{\beta_d}
\end{gather}
with $\alpha_0,\beta_0,\alpha_1,\beta_1,\ldots,\alpha_d,\beta_d\in\R$,
$\alpha_0>0$, $\beta_0>\beta_1>\cdots>\beta_d>0$ and at least one
$\beta_i\not\in\Z$. Since we often only need the leading term we write
$\alpha=\alpha_0$ and $\beta=\beta_0$ for short. They were also able
to show that the discrepancy is $\mathcal{O}((\log N)^{-1})$. We refer
the interested reader to the books of Kuipers and Niederreiter
\cite{kuipers_niederreiter1974:uniform_distribution_sequences}, Drmota
and Tichy \cite{drmota_tichy1997:sequences_discrepancies_and} or
Bugeaud \cite{bugeaud2012:distribution_modulo_one} for a more complete
account on the construction of normal numbers.

The present method of construction by concatenating function values is in
strong connection with properties of $q$-additive functions. We call a
function $f$ strictly $q$-additive, if $f(0)=0$ and the function
operates only on the digits of the $q$-ary representation, i.e.,
\[
  f(n)=\sum_{h=0}^\ell f(d_h)\quad\text{ for }\quad n=\sum_{h=0}^\ell d_hq^h.
\]
A very simple example of a strictly $q$-additive function is the sum of digits
function $s_q$, defined by
\[
  s_q(n)=\sum_{h=0}^\ell d_h\quad\text{ for }\quad n=\sum_{h=0}^\ell d_hq^h.
\]

Refining the methods of Nakai and Shiokawa
\cite{nakai_shiokawa1990:class_normal_numbers} the author obtained the
following result.
\begin{thm}[{\cite[Theorem 1.1]{madritsch2012:summatory_function_q}}]
  Let $q\geq2$ be an integer and $f$ be a strictly $q$-additive
  function. If $p$ is a pseudo-polynomial as defined in
  (\ref{mani:pseudopoly}), then there exists $\eta>0$ such that
\begin{gather*}
  \sum_{n\leq N}f\left(\lf p(n)\rf\right)
  =\mu_fN\log_q(p(N))
  +NF\left(\log_q(p(N))\right)
  +\mathcal{O}\left(N^{1-\eta}\right),
\end{gather*}
where
\[
\mu_f=\frac1q\sum_{d=0}^{q-1}f(d)
\]
and $F$ is a $1$-periodic function depending only on $f$ and $p$.
\end{thm}

In the present paper, however, we are interested in a variant of
$\sigma_q(f)$ involving primes. As a first example, Champernowne
\cite{champernowne1933:construction_decimals_normal} conjectured and
later Copeland and Erd{\H o}s
\cite{copeland_erdoes1946:note_on_normal} proved that the number
\begin{align*}
0.2\,3\,5\,7\,11\,13\,17\,19\,23\,29\,31\,37\,41\,43\,47\,53\,59\,61\,67\dots
\end{align*}
is normal in base $10$. Similar to the construction above we want to
consider the number
\begin{gather*}
\tau_q=\tau_q(f)=0.\lf f(2)\rf_q \lf f(3)\rf_q \lf f(5)\rf_q \lf f(7)\rf_q \lf f(11)\rf_q \lf f(13)\rf_q \dots,
\end{gather*}
where the arguments of $f$ run through the sequence of primes.

Then the paper of Copeland and Erd{\H o}s corresponds to the function
$f(x)=x$. Nakai and Shiokawa
\cite{nakai_shiokawa1997:normality_numbers_generated} showed that the
discrepancy for polynomials having rational coefficients is
$\mathcal{O}((\log N)^{-1})$. Furthermore Madritsch, Thuswaldner and
Tichy
\cite{madritsch_thuswaldner_tichy2008:normality_numbers_generated}
showed, that transcendental entire functions of bounded logarithmic
order yield normal numbers. Finally in a recent paper Madritsch and
Tichy \cite{madritsch_tichy2013:construction_normal_numbers}
considered pseudo-polynomials of the special form $\alpha x^\beta$
with $\alpha>0$, $\beta>1$ and $\beta\not\in\Z$.

The aim of the present paper is to extend this last construction to
arbitrary pseudo-polynomials. Our first main result is the following
\begin{thm}\label{thm:normal}
Let $f$ be a pseudo-polynomial as in (\ref{mani:pseudopoly}). Then
\[
\mathcal{R}_N(\tau_q(f))\ll(\log N)^{-1}.
\]
\end{thm}


In our second main result we use the connection of this construction
of normal numbers with the arithmetic mean of $q$-additive functions
as described above. Known results are due to Shiokawa
\cite{shiokawa1974:sum_digits_prime} and Madritsch and Tichy
\cite{madritsch_tichy2013:construction_normal_numbers}. Similar
results concerning the moments of the sum of digits function over
primes have been established by K\'atai
\cite{katai1977:sum_digits_primes}.

Let $\pi(x)$ stand for the number of primes less than or equal to
$x$. Then adapting these ideas to our method we obtain the following
\begin{thm}\label{thm:summatoryfun}
Let $f$ be a pseudo-polynomial as in (\ref{mani:pseudopoly}). Then 
\[
\sum_{p\leq P}s_q(\lf f(p)\rf)=\frac{q-1}2\pi(P)\log_qP^\beta+\mathcal{O}(\pi(P)),
\]
where the sum runs over the primes and the implicit $\mathcal{O}$-constant may
depend on $q$ and $\beta$.
\end{thm}

\begin{rem}
With simple modifications Theorem \ref{thm:summatoryfun} can be extended to
completely $q$-additive functions replacing $s_q$.
\end{rem}


The proof of the two theorems is divided into four parts. In the
following section we rewrite both statements in order to obtain as a
common base the central theorem -- Theorem \ref{mani:centralthm}. In
Section~\ref{sec:proof-prop-refm1} we start with the proof of this
central theorem by using an indicator function and its Fourier
series. These series contain exponential sums which we treat by
different methods (with respect to the position in the expansion) in
Section \ref{sec:expon-sum-estim}. Finally, in
Section~\ref{sec:proof-prop-refm2} we put the estimates together in
order to proof the central theorem and therefore our two statements.

\section{Preliminaries}\label{sec:preliminaries}

Throughout the rest $p$ will always denote a prime. The implicit
constant of $\ll$ and $\mathcal{O}$ may depend on the
pseudo-polynomial $f$ and on the parameter
$\varepsilon>0$. Furthermore we fix a block $d_1\cdots d_\ell$ of
length $\ell$ and $N$, the number of digits we consider.

In the first step we want to know in the
expansion of which prime the $N$-th digit occurs. This can be seen as
the translation from the digital world to the world of blocks. To this
end let $\ell(m)$ denote the length of the $q$-ary
expansion of an integer $m$. Then we define an integer $P$ by
\begin{gather*}
\sum_{p\leq P-1}\ell\left(\lfloor f(p)\rfloor\right) <N\leq
\sum_{p\leq P}\ell\left(\lfloor f(p)\rfloor\right),
\end{gather*}
where the sum runs over all primes. Thus we get the following relation
between $N$ and $P$
\begin{equation}\label{mani:NtoP}
\begin{split}
N&=\sum_{p\leq P}\ell(\lf f(p)\rf)+\mathcal{O}(\pi(P))+\mathcal{O}(\beta \log_q(P))\\
&=\frac{\beta}{\log q}P+\mathcal{O}\left(\frac{P}{\log P}\right).
\end{split}\end{equation}
Here we have used the prime number theorem in the form (\textit{cf.}
\cite[Th\'eor\`eme 4.1]{tenenbaum1995:introduction_la_theorie}) 
\begin{gather}\label{pnt}
  \pi(x)=\mathrm{Li}\, x+\mathcal{O}\left(\frac x{(\log x)^G}\right),
\end{gather}
where $G$ is an arbitrary positive constant and
\[
  \mathrm{Li}\,x=\int_2^x\frac{\mathrm{d}t}{\log t}.
\]

Now we show that we may neglect the occurrences of the block
$d_1\cdots d_\ell$ between two expansions. We write
$\mathcal{N}(f(p))$ for the number of occurrences of this block in the
$q$-ary expansion of $\lfloor f(p)\rfloor$. Then \eqref{mani:NtoP}
implies that
\begin{gather}\label{mani:Ntrunc}
  \lvert\mathcal{N}(\tau_q(f);d_1\cdots d_\ell;N)-\sum_{p\leq
    P}\mathcal{N}(f(p))\rvert\ll\frac N{\log N}.
\end{gather}

In the next step we use the polynomial-like behavior of $f$. In
particular, we collect all the values having the same length of
expansion. Let $j_0$ be a sufficiently large integer. Then for
each integer $j\geq j_0$ there exists a $P_j$ such that
\[
  q^{j-2}\leq f(P_j)<q^{j-1}\leq f(P_j+1)<q^j
\]
with
\[
  P_j\asymp q^{\frac j\beta}.
\]
Furthermore we set $J$ to be the greatest length of the $q$-ary
expansions of $f(p)$ over the primes $p\leq P$, i.e.,
\begin{gather*}
J:=\max_{p\leq P}\ell(\lfloor f(p)\rfloor)=\log_q(f(P))+\mathcal{O}(1)\asymp\log
P.
\end{gather*}

Now we show that we may suppose that each expansion has the same
length (which we reach by adding leading zeroes). For $P_{j-1}<p\leq
P_j$ we may write $f(p)$ in $q$-ary expansion, i.e.,
\begin{gather}\label{mani:expansionoffp}
f(p)=b_{j-1}q^{j-1}+b_{j-2}q^{j-2}+\dots+b_{1}q+b_{0}+b_{-1}q^{-1}+\dots.
\end{gather}
Then we denote by $\mathcal{N}^*(f(p))$ the number of occurrences of the block
$d_1\cdots d_\ell$ in the string $0\cdots0b_{j-1}b_{j-2}\cdots b_1b_0$, where we
filled up the expansion with leading zeroes such that it has length $J$. The error of
doing so can be estimated by 
\begin{align*}
0&\leq\sum_{p\leq P}\mathcal{N}^*(f(p))-\sum_{p\leq P}\mathcal{N}(f(p))\\
&\leq\sum_{j=j_0+1}^{J-1}(J-j)\left(\pi(P_{j+1})-\pi(P_{j})\right)+\mathcal{O}(1)\\
&\leq\sum_{j=j_0+2}^{J}\pi(P_{j})+\mathcal{O}(1)\ll\sum_{j=j_0+2}^{J}\frac{q^{j/\beta}}j
\ll\frac P{\log P}\ll\frac N{\log N}.
\end{align*}

In the following three sections we will estimate this sum of indicator
functions $\mathcal{N}^*$ in order to prove the following theorem.
\begin{thm}\label{mani:centralthm}
Let $f$ be a pseudo polynomial as in \eqref{mani:pseudopoly}. Then
\begin{gather}\label{mani:centralthm:statement}
\sum_{p\leq
  P}\mathcal{N}^*\left(\lf f(p)\rf\right)=q^{-\ell}\pi(P)\log_qP^\beta+\mathcal{O}\left(\frac{P}{\log
  P}\right)
\end{gather}
\end{thm}

Using this theorem we can simply deduce our two main results.

\begin{proof}[Proof of Theorem \ref{thm:normal}]
We insert \eqref{mani:centralthm:statement} into \eqref{mani:Ntrunc}
and get the desired result.
\end{proof}

\begin{proof}[Proof of Theorem \ref{thm:summatoryfun}]
For this proof we have to rewrite the statement. In particular, we use that the
sum of digits function counts the number of $1$s, $2$s, etc. and
assigns weights to them, i.e.,
\[
s_q(n)=\sum_{d=0}^{q-1}d\cdot\mathcal{N}(n;d).
\]
Thus
\begin{align*}
\sum_{p\leq P}s_q(\lf p^\beta\rf)
&=\sum_{p\leq P}\sum_{d=0}^{q-1}d\cdot\mathcal{N}(p^\beta)
=\sum_{p\leq
  P}\sum_{d=0}^{q-1}d\cdot\mathcal{N}^*(p^\beta)+\mathcal{O}\left(\frac{P}{\log
    P}\right)\\
&=\frac{q-1}2\pi(P)\log_q(P^\beta)+\mathcal{O}\left(\frac{P}{\log
    P}\right)
\end{align*}
and the theorem follows.
\end{proof}


In the following sections we will prove Theorem \ref{mani:centralthm}
in several steps. First we use the ``method of little glasses'' in
order to approximate the indicator function by a Fourier series having
smooth coefficients. Then we will apply different methods (depending
on the position in the expansion) for the estimation of the
exponential sums that appear in the Fourier series. Finally we put
everything together and get the desired estimate.

\section{Proof of Theorem \ref{mani:centralthm}, Part I}\label{sec:proof-prop-refm1}
We want to ease notation by splitting the pseudo-polynomial $f$ into a polynomial and the rest. Then
there exists a unique decomposition of the following form:
\begin{gather}\label{pseudo:poly:split}
f(x)=g(x)+h(x)
\end{gather}
where $h\in\R[X]$ is a polynomial of degree $k$ (where we set
$k=0$ if $h$ is the zero polynomial) and
$$g(x)=\sum_{j=1}^r\alpha_jx^{\theta_j}$$
with $r\geq1$, $\alpha_r\neq0$, $\alpha_j$ real, $0<\theta_1<\cdots<\theta_r$
and $\theta_j\not\in\Z$ for $1\leq j\leq r$.

Let $\gamma$ and $\rho$ be two parameter which we will frequently
use in the sequel. We suppose that
\begin{gather*}
0<\gamma<\rho<\min\left(\frac1{4(k+1)},\frac{\theta_r}{2}\right).
\end{gather*}

The aim of this section is to calculate the Fourier transform of
$\mathcal{N}^*$. In order to count the occurrences of the block
$d_1\cdots d_\ell$ in the $q$-ary expansion of $\lfloor f(p) \rfloor$
($2\le p \le P$) we define the indicator function
\begin{align*}
\mathcal{I}(t)=\begin{cases}
  1, &\text{if }\sum_{i=1}^\ell d_iq^{-i}\leq t-\lfloor t\rfloor
     <\sum_{i=1}^\ell d_iq^{-i}+q^{-\ell};\\
  0, &\text{otherwise;}
     \end{cases}
\end{align*}
which is a $1$-periodic function. Indeed, we have
\begin{gather}\label{position}
\mathcal{I}(q^{-j}f(p)) = 1 \Longleftrightarrow d_1\cdots d_\ell =
b_{j-1}\cdots b_{j-\ell},
\end{gather}
where $f(p)$ has an expansion as in (\ref{mani:expansionoffp}). Thus
we may write our block counting function as follows
\begin{gather}\label{mani:NthetatoNstar}
\mathcal{N}^*(f(p))=\sum_{j=\ell}^J\mathcal{I}\left(q^{-j}f(p)\right).
\end{gather}


In the following we will use Vinogradov's ``method of little glasses''
(\textit{cf.} \cite{vinogradov2004:method_trigonometrical_sums}). We want to approximate
$\mathcal{I}$ from above and from below by two $1$-periodic functions
having small Fourier coefficients. To this end we will use the
following
\begin{lem}[{\cite[Lemma
    12]{vinogradov2004:method_trigonometrical_sums}}]\label{vin:lem12}
Let $\alpha$, $\beta$, $\Delta$ be real numbers satisfying
\begin{gather*}
0<\Delta<\frac12,\quad\Delta\leq\beta-\alpha\leq1-\Delta.
\end{gather*}
Then there exists a periodic function $\psi(x)$ with period $1$,
satisfying
\begin{enumerate}
\item $\psi(x)=1$ in the interval $\alpha+\frac12\Delta\leq x
  \leq\beta-\frac12\Delta$,
\item $\psi(x)=0$ in the interval $\beta+\frac12\Delta\leq x
  \leq1+\alpha-\frac12\Delta$,
\item $0\leq\psi(x)\leq1$ in the remainder of the interval
  $\alpha-\frac12\Delta\leq x\leq1+\alpha-\frac12\Delta$,
\item $\psi(x)$ has a Fourier series expansion of the form
  $$
  \psi(x)=\beta-\alpha+\sum_{\substack{\nu=-\infty\\\nu\neq0}}^\infty
   A(\nu) e(\nu x),
  $$
  where
  \begin{gather}\label{mani:A}
  \lvert A(\nu)\rvert \ll \min \left( \frac 1\nu,
  \beta-\alpha,\frac{1}{\nu^2\Delta} \right).
  \end{gather}
\end{enumerate}
\end{lem}

We note that we could have used Vaaler polynomials
\cite{vaaler1985:some_extremal_functions}, however, we do not gain
anything by doing so as the estimates we get are already best
possible. Setting
\begin{equation}\label{mani:abd}
\begin{split}
\delta=P^{-\gamma},\quad
\begin{aligned}
\alpha_-&=\sum_{\lambda=1}^\ell d_\lambda q^{-\lambda}+(2\delta)^{-1},&
\beta_-&=\sum_{\lambda=1}^\ell d_\lambda q^{-\lambda}+q^{-\ell}-(2\delta)^{-1},\\
\alpha_+&=\sum_{\lambda=1}^\ell d_\lambda q^{-\lambda}-(2\delta)^{-1},&
\beta_+&=\sum_{\lambda=1}^\ell d_\lambda q^{-\lambda}+q^{-\ell}+(2\delta)^{-1}.
\end{aligned}
\end{split}
\end{equation}
and an application of Lemma \ref{vin:lem12} with
$(\alpha,\beta,\delta)=(\alpha_-,\beta_-,\delta)$ and
$(\alpha,\beta,\delta)=(\alpha_+,\beta_+, \delta)$,
respectively, provides us with two functions $\mathcal{I}_-$ and
$\mathcal{I}_+$. By our choice of
$(\alpha_\pm,\beta_\pm,\delta)$ it is immediate that
\begin{equation}\label{uglI}
\mathcal{I}_-(t)\leq\mathcal{I}(t)\leq\mathcal{I}_+(t) \qquad
(t\in\mathbb{R}).
\end{equation}
Lemma \ref{vin:lem12} also implies that these two functions have
Fourier expansions
\begin{align}\label{mani:Ifourier}
\mathcal{I}_\pm(t)=q^{-\ell}\pm P^{-\gamma}+
  \sum_{\substack{\nu=-\infty\\\nu\neq0}}^\infty A_\pm(\nu)e(\nu t)
\end{align}
satisfying
\begin{gather*}
\lvert A_\pm(\nu)\rvert
\ll\min(\lvert\nu\rvert^{-1},P^{\gamma}\lvert\nu\rvert^{-2}).
\end{gather*}
In a next step we want to replace $\mathcal{I}$ by $\mathcal{I}_+$
in (\ref{mani:NthetatoNstar}). For this purpose we observe, using \eqref{uglI},
and \eqref{mani:Ifourier} that
\begin{gather*}
\lvert\mathcal{I}(t)-q^{-\ell}\rvert
  \ll P^{-\gamma} + \sum_{\substack{\nu=-\infty\\\nu\neq0}}^\infty
  A_\pm(\nu)e(\nu t).
\end{gather*}
Thus setting $t=q^{-j}f(p)$ and summing over $p\leq P$ yields
\begin{gather}\label{mani:0.5}
\lvert\sum_{p\leq P}\mathcal{I}(q^{-j}f(p))-\frac{\pi(P)}{q^{\ell}}\rvert
\ll\pi(P)P^{-\gamma}+\sum_{\substack{\nu=-\infty\\\nu\neq0}}^\infty
A_{\pm}(\nu)\sum_{p\leq P}e\left(\frac{\nu}{q^j}f(p)\right).
\end{gather}

Now we consider the coefficients $A_\pm(\nu)$. Noting
\eqref{mani:A} one observes that
\begin{gather*}
A_\pm(\nu)\ll\begin{cases}
  \nu^{-1},       &\text{for }\lvert\nu\rvert\leq P^{\gamma};\\
  P^{\gamma}\nu^{-2}, &\text{for }\lvert\nu\rvert>P^{\gamma}.
          \end{cases}
\end{gather*}
Estimating all summands with $\lvert\nu\rvert>P^{\gamma}$ trivially we get
\begin{gather*}
\sum_{\substack{\nu=-\infty\\\nu\neq0}}^\infty
  A_\pm(\nu)e\left(\frac{\nu}{q^j}f(p)\right)
\ll\sum_{\nu=1}^{P^{\gamma}}\nu^{-1}e\left(\frac{\nu}{q^j}f(p)\right)+P^{-\gamma}.
\end{gather*}
Using this in \eqref{mani:0.5} yields
\begin{gather*}
\lvert\sum_{p\leq P}\mathcal{I}(q^{-j}f(p))-\frac{\pi(P)}{q^{\ell}}\rvert
\ll\pi(P)P^{-\gamma}+\sum_{\nu=1}^{P^{\gamma}}
\nu^{-1}S(P,j,\nu),
\end{gather*}
where we have set 
\begin{gather}\label{S_Pjnu}
S(P,j,\nu):=\sum_{p\leq P}e\left(\frac{\nu}{q^j}f(p)\right).
\end{gather}

\section{Exponential sum estimates}\label{sec:expon-sum-estim}
In the present section we will focus on the estimation of the sum
$S(P,j,\nu)$ for different ranges of $j$. Since $j$ describes the
position within the $q$-ary expansion of $f(p)$ we will call these
ranges the ``most significant digits'', the ``least significant
digits'' and the ``digits in the middle'', respectively.

Now, if $\theta_r>k\geq0$, \textit{i.e} the leading coefficient of $f$ origins
from the pseudo polynomial part $g$, then we consider the two ranges
$$1\leq q^j\leq P^{\theta_r-\rho}\quad\text{and}\quad
P^{\theta_r-\rho}<q^j\leq P^{\theta_r}.$$
For the first one we will apply Proposition \ref{prop:least_significant} and
for the second one Proposition \ref{prop:most_significant}.

On the other hand, if $k>\theta_r>0$, meaning that the leading coefficient of
$f$ origins from the polynomial part $h$, then we have an additional
part. In particular, in this case we will consider the three ranges
$$1\leq q^j\leq P^{\theta_r-\rho},\quad
P^{\theta_r-\rho}<q^j\leq P^{k-1+\rho},\quad\text{and}\quad
P^{k-1+\rho}<q^j\leq P^{k}.$$ We will, similar to above, treat the
first and last range by Proposition~\ref{prop:least_significant} and
Proposition \ref{prop:most_significant}, respectively. For the middle
range we will apply Proposition \ref{prop:middle_digits}. Since
$2\rho<\theta_r$, we note that the middle range is empty if $k=1$.

Since the size of $j$ represents the position of the digit in the
expansion (\textit{cf.} \eqref{position}), we will deal in the
following subsection with the ``most significant digits'', the ``least
significant digits'' and the ``digits in the middle'', respectively.

\subsection{Most significant digits}
We start our series of estimates for the exponential sum $S(P,j,\nu)$
for $j$ being in the highest range. In particular, we want to show the
following
\begin{prop}\label{prop:most_significant}
Suppose that for some $k\geq1$ we have $\lvert
f^{(k)}(x)\rvert\geq\Lambda$ for any $x$ on $[a,b]$ with
$\Lambda>0$. Then
$$S(P,j,\nu)\ll\frac1{\log P}\Lambda^{-\frac1k}+\frac{P}{(\log P)^G}.$$
\end{prop}

The main idea of the proof is to use Riemann-Stieltjes integration
together with 
\begin{lem}[{\cite[Lemma 8.10]{iwaniec_kowalski2004:analytic_number_theory}}]
\label{ik:lem8.10}
Let $F\colon[a,b]\to\R$ and suppose that for some $k\geq1$
we have $\lvert F^{(k)}(x)\rvert\geq\Lambda$ for any $x$ on $[a,b]$
with $\Lambda>0$. Then
\[
\lvert\int_a^be(F(x))\mathrm{d}x\rvert
\leq k2^k\Lambda^{-1/k}.
\]
\end{lem}


\begin{proof}[Proof of Proposition \ref{prop:most_significant}]
We rewrite the sum into a Riemann-Stieltjes integral:
\begin{align*}
S(P,j,\nu)=\sum_{p\leq P}e\left(\frac{\nu}{q^j}f(p)\right)
=\int_{2}^{P}e\left(\frac{\nu}{q^j}f(t)\right)\mathrm{d}\pi(t)+\mathcal{O}(1).
\end{align*}
Then we apply the prime number theorem in the form
\eqref{pnt} to gain the usual integral back. Thus
\begin{align*}
S(P,j,\nu)
=\int_{P(\log P)^{-G}}^{P}
e\left(\frac{\nu}{q^j}f(t)\right)
\frac{\mathrm{d}t}{\log t}
+\mathcal{O}\left(\frac{P}{(\log P)^G}\right).
\end{align*}
Now we use the second mean-value theorem to get
\begin{equation}\label{mani:res:most}
\begin{split}
S(P,j,\nu)\ll\frac1{\log P}\sup_{\xi}
  \lvert\int_{P(\log P)^{-G}}^{\xi}e\left(\frac{\nu}{q^j}f(t)\right)\mathrm{d}t\rvert
  +\frac{P}{(\log P)^G}.
\end{split}
\end{equation}
Finally an application of Lemma \ref{ik:lem8.10} proves the lemma.
\end{proof}

\subsection{Least significant digits} Now we turn our attention to the
lowest range of $j$. In particular, the goal is the proof of the
following
\begin{prop}\label{prop:least_significant}
Let $P$ and $\rho$ be positive reals and $f$ be a pseudo-polynomial as in
\eqref{pseudo:poly:split}. If $j$ is such that
\begin{gather}\label{mani:gammarange}
  1\leq q^j\leq P^{\theta_r-\rho}
\end{gather}
holds, then for $1\leq\nu\leq P^\gamma$ there exists $\eta>0$
(depending only on $f$ and $\rho$) such that
\begin{gather*}
  S(P,j,\nu)=(\log P)^8P^{1-\eta}.
\end{gather*}
\end{prop}


Before we launch into the proof we collect some tools that will be
necessary in the sequel. A standard idea for estimating exponential
sums over the primes is to rewrite them into ordinary exponential sums
over the integers having von Mangoldt's function as weights and then
to apply Vaughan's identity. We denote by
\[
\Lambda(n)=\begin{cases}
\log p,&\text{if $n=p^k$ for some prime $p$ and an integer $k\geq1$;}\\
0,&\text{otherwise}.
\end{cases}
\]
von Mangoldt's function. For the rewriting process we use the following
\begin{lem}
\label{mr:lem11}
Let $g$ be a function such that $\lvert g(n)\rvert\leq 1$ for all
integers $n$. Then
\[
\lvert\sum_{p\leq P}g(p)\rvert\ll\frac1{\log P}\max_{t\leq P}
\lvert\sum_{n\leq t}\Lambda(n)g(n)\rvert+\mathcal{O}(\sqrt{P}).
\]
\end{lem}

\begin{proof}
This is Lemma 11 of \cite{mauduit_rivat2010:sur_un_probleme}. However,
the proof is short and we need some piece later.

We start with a summation by parts yielding
\[\sum_{p\leq P}g(p)=\frac1{\log P}\sum_{p\leq
  x}\log(p)g(p)+\int_2^P\left(\sum_{p\leq
    t}\log(p)g(p)\right)\frac{\mathrm{d}t}{t(\log t)^2}.\]

Now we cut the integral at $\sqrt{P}$ and use Chebyshev's inequality
(\textit{cf.} \cite[Th\'eor\`eme 1.3]{tenenbaum1995:introduction_la_theorie})
in the form $\sum_{p\leq
  t}\log(p)\leq\log(t)\pi(t)\ll t$ for the lower part. Thus
\begin{align*}
\lvert\sum_{p\leq P}g(p)\rvert
&\leq\left(\frac1{\log
    P}+\int_{\sqrt{P}}^P\frac{\mathrm{d}t}{t(\log t)^2}\right)
  \max_{\sqrt{P}<t\leq P}\lvert\sum_{p\leq
    P}\log(p)g(p)\rvert+\mathcal{O}(\sqrt{P})\\
&=\frac2{\log P}\max_{\sqrt{P}<t\leq P}\lvert\sum_{p\leq t}\log(p)g(p)\rvert+\mathcal{O}(\sqrt{P}).
\end{align*}

Finally we again use Chebyshev's inequality $\pi(t)\ll t/\log(t)$ to obtain
\begin{gather}\label{mani:log_Mangoldt_equivalence}
\lvert\sum_{n\leq t}\Lambda(n)g(n)-\sum_{p\leq t}\log(p)g(p)\rvert
\leq\sum_{p\leq\sqrt{t}}\log(p)\sum_{a=2}^{\lf\frac{\log(t)}{\log(p)}\rf}1
\leq\pi(\sqrt{t})\log(t)\ll\sqrt{t}.
\end{gather}
\end{proof}

In the next step we use Vaughan's identity to subdivide this weighted
exponential sum into several sums of Type I and II.
\begin{lem}[{\cite[Lemma
    2.3]{bergelson_kolesnik_madritsch+2014:uniform_distribution_prime}}]
\label{bkmst:lem23}
Assume $F(x)$ to be any function defined on the real line, supported on $[P/2, P]$
and bounded by $F_0$. Let further $U,V,Z$ be any parameters satisfying $3
\leq U < V < Z < P$, $Z \geq 4U^2$, $P \geq 64 Z^2 U$, $V^3 \geq 32 P$ and
$Z-\frac12\in\mathbb{N}$. Then
$$\left| \sum_{P/2< n\leq P} \Lambda(n) F(n)  \right| \ll K
\log P + F_0 +  L (\log P)^8 ,$$
where $K$ and $L$ are defined by
\begin{align*}
K&=\max_M\sum_{m=1}^\infty d_3(m)\left\vert\sum\limits_{Z<n\leq M} F(mn)\right\vert,\\
L&=\sup\sum_{m=1}^\infty d_4(m)\left\vert\sum\limits_{U < n < V} b(n) F(mn)\right\vert,
\end{align*}
where the supremum is taken over all arithmetic functions $b(n)$ satisfying $|b(n)| \leq d_3(n).$
\end{lem}

After subdividing the weighted exponential sum with Vaughan's identity
we will use the following lemma in order to estimate the occurring
exponential sums.
\begin{lem}[{\cite[Lemma
     2.5]{bergelson_kolesnik_madritsch+2014:uniform_distribution_prime}}]
\label{bkmst:lem25}
   Let $X,k,q\in \mathbb{N}$ with $k,q\geq 0$ and set $K=2^k$ and $Q=
   2^q$. Let $h(x)$ be a polynomial of degree $k$ with real
   coefficients. Let $g(x)$ be a real $(q+k+2)$ times continuously
   differentiable function on $[X/2 , X]$ such that $\left| f^{(r)}(x)
   \right| \asymp F X^{-r}$ $( r = 1, \dots, q+k+2) $.  Then, if $F =
   o (X^{q+2})$ for $F$ and $X$ large enough, we have
 $$\left| \sum_{X/2 < x \leq X} e(g(x) + h(x)) \right| \ll X^{1 - \frac{1}{K}} + X \left( \frac{\log^k X}{F} \right)^{\frac{1}{K}} + X \left( \frac{F}{X^{q+2}} \right)^{\frac{1}{(4KQ-2K)}}.$$
\end{lem}

Now we have the necessary tools to state the
\begin{proof}[Proof of Proposition \ref{prop:least_significant}]
An application of Lemma \ref{mr:lem11} yields
\[S(P,j,\nu)\ll\frac1{\log P}\max\lvert\sum_{n\leq
  P}\Lambda(n)e\left(\frac{\nu}{q^j}(g(n)+h(n))\right)\rvert+P^{\frac12}.\]
We split the inner sum into $\leq \log
P$ sub sums of the form $$\lvert \sum\limits_{X< n \leq
  2X}\Lambda(n)e\left(\frac{\nu}{q^j}(g(n)+h(n))\right)\rvert$$ with
$2X \leq P$ and let $S$ be a typical one of them. We may
assume that $X \geq P^{1-\rho}$.

Using Vaughan's identity (Lemma~\ref{bkmst:lem23}) with $U = \frac{1}{4} X^{1/5}$, $V= 4 X^{1/3}$ and $Z$ the
unique number in $\frac12+\mathbb{N}$, which is closest to $\frac{1}{4}
X^{2/5}$, we obtain
\begin{gather}\label{mani:S}
S \ll 1+(\log X)S_1+(\log X)^8S_2,
\end{gather}
where
\begin{align*}
S_1&=\sum_{x < \frac{2X}{Z}} d_3(x) \sum_{y > Z, \frac{X}{x} < y < \frac{2X}{x}} e\left(\frac{\nu}{q^j}(g(xy)+h(xy))\right)\\
S_2&=\sum_{\frac XV<x\leq\frac{2X}U} d_4(x) \sum_{U < y < V, \frac{X}{x} < y \leq \frac{2X}{x}} b(y) e\left(\frac{\nu}{q^j}(g(xy)+h(xy))\right)\notag
\end{align*}

We start with the estimation of $S_1$. Since $d_3(x)\ll
x^{\varepsilon}$ we have for
\begin{align*}
\lvert S_1\rvert\ll X^\varepsilon\sum_{x\leq\frac{2X}Z}
  \lvert\sum_{\substack{\frac Xx<y\frac{2X}x\\y>Z}}e\left(\frac{\nu}{q^j}(g(xy)+h(xy))\right)\rvert.
\end{align*}
For estimating the inner sum we fix $x$ and denote $Y=\frac Xx$. Since
$\theta_r\not\in\Z$ and $\theta_r>k\geq0$, we have that
\[\lvert\frac{\partial^\ell g(xy)}{\partial y^\ell}\rvert
\asymp X^{\theta_r}Y^{-\ell}.\]

Now on the one hand, since $q^j\leq P^{\theta_r-\rho}$, we have $\nu
q^{-j}X^{\theta_r}\gg X^{\rho}$. On the other hand for
$\ell\geq5(\lfloor\theta_r\rfloor+1)$ we get
\[\frac{\nu}{q^j}X^{\theta_r}Y^{-\ell}\leq P^\gamma X^{\theta_r-\frac25\ell}\ll X^{-\frac12}.\] 

Thus an application of Lemma \ref{bkmst:lem25} yields the following
estimate:
\begin{equation}\label{mani:estim:S1}
\begin{split}
\lvert S_1\rvert &\ll X^{\varepsilon}\sum_{x \leq 2X/Z} Y \left[
  Y^{-\frac{1}{K}} + (\log Y)^kX^{-\frac{\rho}{K}} + X^{-\frac{1}{2} \frac{1}{4K \cdot 8L^5 - 2K}} \right] \\
&\ll X^{1+\varepsilon}(\log X)\left(X^{-\rho} + X^{-\frac{1}{64L^5-4} } \right)^{\frac1K}, 
\end{split}\end{equation}
where we have used that $\frac kK<1$ and $\rho<\frac13$.

For the second sum $S_2$ we start by splitting the interval $(
\frac{X}{V} , \frac{2X}{U} ]$ into $\leq \log X$ subintervals of the
form $(X_1, 2X_1]$. Thus
\begin{align*}
\lvert S_2\rvert
&\leq (\log X)X^{\varepsilon}\sum_{X_1<x\leq
  2X_1}\lvert\sum_{\substack{U<y<V\\\frac
    Xx<y\leq\frac{2X}x}}b(y)e\left(\frac{\nu}{q^j}(g(xy)+h(xy))\right)\rvert
\end{align*}

Now an application of Cauchy's inequality together with $\lvert
b(y)\rvert\ll X^\varepsilon$ yields
\begin{align*}
\lvert S_2\rvert^2
&\leq (\log X)^2X^{2\varepsilon}X_1\sum_{X_1<x\leq
  2X_1}\lvert\sum_{\substack{U<y<V\\\frac
    Xx<y\leq\frac{2X}x}}b(y)e\left(\frac{\nu}{q^j}(g(xy)+h(xy))\right)\rvert^2\\
&\ll (\log X)^2X^{4\varepsilon}X_1\\
&\quad\times\left(X_1\frac{X}{X_1}+\lvert \sum_{X_1<x\leq2X_1} \sum_{A < y_1 < y_2 \leq B}e\left(\frac{\nu}{q^j} (g(xy_1)-g(xy_2) + h(xy_1)-h(xy_2))\right)  \rvert\right)
\end{align*}
where $A = \max \{U, \frac{X}{x} \} $ and
$B = \min \{U, \frac{2X}{x} \}$. Changing the order of summation, we
get
\begin{multline*}
|S_2|^2 \ll (\log X)^2X^{4\varepsilon}X_1\\
\times\left(X+
  \sum_{A < y_1 < y_2 \leq B}\lvert \sum_{X_1<x\leq2X_1} e\left(\frac{\nu}{q^j} (g(xy_1)-g(xy_2) + h(xy_1)-h(xy_2))\right)  \rvert\right)
\end{multline*}

As above we want to apply Lemma \ref{bkmst:lem25}. To this end we fix
$y_1$ and $y_2 \ne y_1$. Similarly to above we get that
$$\lvert\frac{\partial^\ell\left(g(xy_1)-g(xy_2)+h(xy_1)-h(xy_2)\right)}{\partial x^\ell}\rvert\asymp\frac{\lvert
  y_1-y_2\rvert}{y_1}X^{\theta_r}X_1^{-\ell}.$$
Now, on the one hand we have $\frac{\nu}{q^j}\frac{\lvert
  y_1-y_2\rvert}{y_1}X^{\theta_r}\gg X^{\rho}$ and on the other hand
\[\frac{\nu}{q^j}\frac{\lvert
  y_1-y_2\rvert}{y_1}X^{\theta_r}X_1^{-\ell}
\ll X^{\gamma+\theta_r}\left(\frac{X}{V}\right)^{-\ell}
\ll X^{\gamma+\theta_r-\frac23\ell}
\ll X^{-\frac12}\]
if $\ell\geq2\lfloor\theta_r\rfloor+3$. Thus again an application of Lemma \ref{bkmst:lem25} yields
\begin{equation}\label{mani:estim:S2}
\begin{split}
\lvert S_2\rvert^2 &\ll (\log X)^2X^{4\varepsilon}X_1\left(X +
  \sum_{A<y_1<y_2\leq B} X_1\left( X_1^{-\frac{1}{K}} + X^{-\frac{\rho}{K}} + X^{-\frac{1}{2} \frac{1}{4K \cdot 2L^2 - 2K}}\right)\right) \\
&\ll (\log X)^2X^{4\varepsilon}\left(X^{\frac53} + X^{2-\frac{\rho}{K}} + X^{2- \frac{1}{16KL^2 - 4K}}\right).
\end{split}
\end{equation}

Plugging the two estimates \eqref{mani:estim:S1} and
\eqref{mani:estim:S2} into \eqref{mani:S} proves the proposition.
\end{proof}

\subsection{The digits in the middle}

Now we are getting more involved in order to consider those $j$
leading to a position between $\theta_r$ and $k$. These sums
correspond to the ``digits in the middle'' in the proof of Theorem
\ref{mani:centralthm}. We want to prove the following
\begin{prop}\label{prop:middle_digits}
  Let $P$ and $\rho$ be positive reals and $f$ be a pseudo-polynomial as in
  \eqref{pseudo:poly:split}. If $2\rho<\theta_r<k$ and $j$ is such that
\begin{gather}\label{mani:middle_range}
  P^{\theta_r-\rho}< q^j\leq P^{k-1+\rho}
\end{gather}
holds, then for $1\leq\nu\leq P^\gamma$ we have
\begin{gather*}
  S(P,j,\nu)=\sum_{p\leq P}e\left(\frac{\nu f(p)}{q^j}\right)\ll P^{1-\frac{\rho}{4^k}}.
\end{gather*}
\end{prop}

The main idea in this range is to use that the dominant part of $f$
comes from the polynomial $h$. Therefore after getting rid of the
function $g$ we will estimate the sum over the polynomial by the
following
\begin{lem}\label{lem:exponential_sum_primes_poly}
Let $h\in\R[X]$ be a
polynomial of degree $k\geq2$. Suppose $\alpha$ is the leading
coefficient of $h$ and that there are integers $a$, $q$ such that
$$\lvert q\alpha-a\rvert<\frac1q\quad\text{with}\quad
(a,q)=1.$$
Then we have for any $\varepsilon>0$ and $H\leq X$
$$\sum_{X<p\leq X+H}\log(p)e(h(p))\ll
H^{1+\varepsilon}\left(\frac1q+\frac1{H^{\frac12}}+\frac{q}{H^k}\right)^{4^{1-k}}.$$
\end{lem}

\begin{proof}
  This is a slight variant of \cite[Theorem
  1]{harman1981:trigonometric_sums_over}, where we sum over an
  interval of the form $]X,X+H]$ instead of one of the form
  $]0,X]$.
\end{proof}

Now we have enough tools to state the 
\begin{proof}[Proof of Proposition \ref{prop:middle_digits}]
  As in the Proof of Proposition \ref{prop:least_significant} we start
  by an application of Lemma~\ref{mr:lem11} yielding
\[S(P,j,\nu)\ll\frac1{\log P}\max\lvert\sum_{n\leq
  P}\Lambda(n)e\left(\frac{\nu}{q^j}(g(n)+h(n))\right)\rvert+P^{\frac12}.\]

We split the inner sum into $\leq \log P$ sub sums of the form
\[S:=\sum_{X<n\leq X+H}\Lambda(n)e\left(\frac{\nu}{q^j}(g(n)+h(n))\right)\]
with $P^{1-2\rho}\leq X\leq P$ and \[H=\min\left(P^{1-\theta_r}\lvert\nu\rvert^{-1}q^j,X\right).\]
Now we want to separate the function parts $g$ and $h$. Therefore we define two
functions $T$ and $\varphi$ by
$$T(x)=\sum_{X< n\leq
  X+x}\Lambda(n)e\left(\frac{\nu}{q^j}h(n)\right)
\quad\text{and}\quad
\varphi(x):=e\left(\frac{\nu}{q^j}g\left(X+x\right)\right)$$
Then an application of summation by parts yields
\begin{equation}\label{mani:eq_1}
\begin{split}
\sum_{X< n\leq X+H}\Lambda(n)
e\left(\frac{\nu}{q^j}(g(n)+h(n))\right)
&=\sum_{n=1}^{H} \varphi(n)(T(n)-T(n-1))\\
&=\sum_{n=1}^{H}T(n)\left(\varphi(n)-\varphi(n+1)\right)+\varphi(H-1)T(H)\\
&\ll\lvert T(H)\rvert+\sum_{n=1}^{H-1}\left|\varphi(n)-\varphi(n+1)\right|\lvert T(n)\rvert
\end{split}
\end{equation}

Let $\alpha_k$ be the leading coefficient of $P$. Then by Diophantine
approximation there always exists a rational $a/b$ with $b>0$,
$(a,b)=1$,
\[1\leq b\leq H^{k-\rho}
\quad\text{and}\quad
\lvert\frac{\nu\alpha_k}{q^j}-\frac ab\rvert\leq
\frac{H^{\rho-k}}{b}.\]
We distinguish three cases according to the size of $b$.
\begin{itemize}
\item[] \textbf{Case 1.} $H^\rho<b$. In this case we may apply Lemma
  \ref{lem:exponential_sum_primes_poly} together with
  \eqref{mani:log_Mangoldt_equivalence} to get $$T(h)\ll
  H^{1-\frac{\rho}{4^{k-1}}+\varepsilon}.$$
\item[] \textbf{Case 2.} $2\leq b<H^\rho$. In this case we get
  that $$\lvert\frac{\nu\alpha_k}{q^j}\rvert\geq\lvert\frac
  ab\rvert-\frac1{b^2}\geq\frac1{2b}\geq\frac12H^{-\rho}\geq\frac12P^{-\rho}.$$
  Since $2\rho<\theta_r$, this contradicts our lower bound $q^j\geq P^{\theta_r-\rho}$.
\item[] \textbf{Case 3.} $b=1$. This case requires a further
  distinction according to whether $a=0$ or not.
  \begin{itemize}
  \item[] \textbf{Case 3.1.}
    $\lvert\frac{\nu\alpha_k}{q^j}\rvert\geq\frac12$. It follows
    that $$q^j\leq2\lvert\nu\alpha_k\rvert$$ again contradicting our lower
    bound $q^j\geq P^{\theta_r-\rho}$.
  \item[] \textbf{Case 3.2.}
    $\lvert\frac{\nu\alpha_k}{q^j}\rvert<\frac12$. This implies that
    $a=0$ which yields 
    \begin{gather}\label{case3.2}
    q^j\geq\lvert\nu\alpha_k\rvert H^{k-\rho}.
    \end{gather}
    We distinguish two further cases according to
    whether $P^{1-\theta_r}\lvert\nu\rvert^{-1}q^j\leq X$ or not.
    \begin{itemize}
    \item[] \textbf{Case 3.2.1}
      $P^{1-\theta_r}\lvert\nu\rvert^{-1}q^j\leq X$. This implies that
      $q^j\leq P^{\theta_r}\lvert\nu\rvert$
      and \[H=P^{1-\theta_r}\lvert\nu\rvert^{-1}q^j\geq
      P^{1-\rho}\lvert\nu\rvert^{-1}\geq P^{1-2\rho}.\] Plugging these
      estimates into \eqref{case3.2} gives \[P^{\theta_r}\geq
      \lvert\alpha_k\rvert P^{(1-2\rho)(k-\rho)}.\] However, since
      $4(k+1)\rho<1$, we
      have \[(1-2\rho)(k-\rho)>k-1+2\rho\geq\theta_r\] yielding a
      contradiction.
    \item[] \textbf{Case 3.2.2}
      $P^{1-\theta_r}\lvert\nu\rvert^{-1}q^j>X$. Then $H=X\geq
      P^{1-2\rho}$ and \eqref{case3.2}
      becomes \[P^{k-1+\rho}\geq\lvert\nu\alpha_k\rvert
      P^{(1-2\rho)(k-\rho)}\] yielding a similar contradiction as in
      \textbf{Case 3.2.1}.
    \end{itemize}
  \end{itemize}
\end{itemize}
Therefore \textbf{Case 1} is the only possible and we may always apply
Lemma \ref{lem:exponential_sum_primes_poly} together with
\eqref{mani:log_Mangoldt_equivalence}. Plugging this
into~\eqref{mani:eq_1} yields
\begin{align*}
\sum_{X< n\leq X+H}\Lambda(n)
e\left(\frac{\nu}{q^j}(g(n)+h(n))\right)
&\ll H^{1-\frac{\rho}{4^{k-1}}+\varepsilon}\left(1+\sum_{X< n\leq X+H}\left|\varphi(n)-\varphi(n+1)\right|\right)
\end{align*}

Now by our choice of $H$ together with an application of the mean
value theorem we have that
$$\sum_{X\leq n\leq X+H}\lvert \varphi(n)-\varphi(n+1)\rvert
\ll H\frac{\nu}{q^j}P^{\theta-1}\ll 1.$$
Thus 
\begin{align*}
\sum_{X\leq n\leq X+H} \Lambda(n) 
e\left(\frac{\nu}{q^j}(g(n)+h(n))\right)
\ll H^{1-\frac{\rho}{4^{k-1}}+\varepsilon}.
\end{align*}

\end{proof}

\section{Proof of Theorem \ref{mani:centralthm}, Part II}\label{sec:proof-prop-refm2}
Now we use all the tools from the section above in order to estimate

\begin{gather}\label{distance_from_mean}
\sum_{j=\ell}^J\lvert\sum_{p\leq P}\mathcal{I}(q^{-j}f(p))-\frac{\pi(P)}{q^{\ell}}\rvert
\ll\pi(P)H^{-1}J+\sum_{\nu=1}^{H}
\nu^{-1}\sum_{j=\ell}^JS(P,j,\nu).
\end{gather}

As indicated in the section above, we split the sum over $j$ into two
or three parts according to whether $\theta_r>k$ or not. In any case
an application of Proposition \ref{prop:least_significant} yields for the
least significant digits that
\begin{gather}\label{estimate:least}
\sum_{1\leq \nu\leq P^\gamma}\nu^{-1}\sum_{1\leq q^{j}\leq
  P^{\theta_r-\rho}} S(P,j,\nu)
\ll (\log P)^9JP^{1-\eta}.
\end{gather}

Now let us suppose that $\theta_r>k$. Then an application of Proposition
\ref{prop:most_significant} yields
\begin{equation}\label{estimate:most_non_integer}
\begin{split}
\sum_{1\leq \nu\leq P^\gamma}\nu^{-1}&\sum_{P^{\theta_r-\rho}< q^{j}\leq
  P^{\theta_r}}S(P,j,\nu)\\
&\ll \sum_{1\leq \nu\leq P^\gamma}\nu^{-1}\sum_{P^{\theta_r-\rho}< q^{j}\leq
  P^{\theta_r}}\frac1{\log
  P}\left(\frac{\nu}{q^j}\right)^{-\frac1{\lf\theta_r\rf}}+\frac{P}{(\log P)^{G-2}}\\
&\ll \frac{P}{\log P}.
\end{split}
\end{equation}

Plugging the estimates \eqref{estimate:least} and
\eqref{estimate:most_non_integer} into~\eqref{distance_from_mean} we
get that
$$\sum_{j=\ell}^J\lvert\sum_{p\leq P}\mathcal{I}(q^{-j}f(p))-\frac{\pi(P)}{q^{\ell}}\rvert
\ll\frac{P}{\log P},$$
which together with \eqref{mani:NthetatoNstar} proves Theorem
\ref{mani:centralthm} in the case that $\theta_r>k$.

On the other side if $\theta_r<k$, then we consider the two ranges
$$P^{\theta_r-\rho}<q^j\leq P^{k-1+\rho}
\quad\text{and}\quad
P^{k-1+\rho}<q^j\leq P^k.$$
For the ``digits in the middle'' an application of Proposition
\ref{prop:middle_digits} yields
\begin{equation}\label{estimate:middle}
\begin{split}
\sum_{1\leq \nu\leq P^\gamma}\nu^{-1}\sum_{P^{\theta_r-\rho}< q^{j}\leq
  P^{k-1+\rho}}S(P,j,\nu)
&\ll \sum_{1\leq \nu\leq P^\gamma}\nu^{-1}\sum_{P^{\theta_r-\rho}< q^{j}\leq
  P^{k-1+\rho}}P^{1-\frac{\rho}{4^k}}\\
&\ll(\log P)JP^{1-\frac{\rho}{4^k}}.
\end{split}
\end{equation}

Finally we consider the most significant digits. By an application of
Proposition~\ref{prop:most_significant} we have
\begin{equation}\label{estimate:most_integer}
\begin{split}
\sum_{1\leq \nu\leq P^\gamma}\nu^{-1}&\sum_{P^{k-1+\rho}< q^{j}\leq
  P^{k}}S(P,j,\nu)\\
&\ll \sum_{1\leq \nu\leq P^\gamma}\nu^{-1}\sum_{P^{k-1+\rho}< q^{j}\leq
  P^{k}}\frac1{\log
  P}\left(\frac{\nu}{q^j}\right)^{-\frac1k}+\frac{P}{(\log P)^{G-2}}\\
&\ll \frac{P}{\log P}.
\end{split}
\end{equation}

Plugging the estimates \eqref{estimate:least}, \eqref{estimate:middle} and
\eqref{estimate:most_integer} into~\eqref{distance_from_mean} we
get that
$$\sum_{j=\ell}^J\lvert\sum_{p\leq P}\mathcal{I}(q^{-j}f(p))-\frac{\pi(P)}{q^{\ell}}\rvert
\ll\frac{P}{\log P},$$
which together with \eqref{mani:NthetatoNstar} proves Theorem
\ref{mani:centralthm} in the case that $\theta_r<k$.

\section*{Acknowledgment}
The author wants to thank G{\'e}rald Tenenbaum for many fruitful
discussions and suggestions in connection with the proof of
Proposition \ref{prop:middle_digits}.


\providecommand{\bysame}{\leavevmode\hbox to3em{\hrulefill}\thinspace}
\providecommand{\MR}{\relax\ifhmode\unskip\space\fi MR }
\providecommand{\MRhref}[2]{%
  \href{http://www.ams.org/mathscinet-getitem?mr=#1}{#2}
}
\providecommand{\href}[2]{#2}


\begin{thebibliography}{10}

\bibitem{bergelson_kolesnik_madritsch+2014:uniform_distribution_prime}
V.~Bergelson, G.~Kolesnik, M.~Madritsch, Y.~Son, and R.~Tichy, \emph{Uniform
  distribution of prime powers and applications to van der corput sets}, Israel
  Journal of Mathematics (2013), accepted.

\bibitem{borel1909:les_probabilites_denombrables}
E.~Borel, \emph{{Les probabilit\'es d\'enombrables et leurs applications
  arithm\'etiques.}}, Palermo Rend. \textbf{27} (1909), 247--271 (French).

\bibitem{bugeaud2012:distribution_modulo_one}
Y.~Bugeaud, \emph{Distribution modulo one and {D}iophantine approximation},
  Cambridge Tracts in Mathematics, vol. 193, Cambridge University Press,
  Cambridge, 2012.

\bibitem{champernowne1933:construction_decimals_normal}
D.~Champernowne, \emph{{The construction of decimals normal in the scale of
  ten}}, J. Lond. Math. Soc. \textbf{8} (1933), 254--260 (English).

\bibitem{copeland_erdoes1946:note_on_normal}
A.~H. Copeland and P.~Erd{\H o}s, \emph{Note on normal numbers}, Bull. Amer.
  Math. Soc. \textbf{52} (1946), 857--860.

\bibitem{davenport_erdoes1952:note_on_normal}
H.~Davenport and P.~Erd{\H o}s, \emph{Note on normal decimals}, Canadian J.
  Math. \textbf{4} (1952), 58--63.

\bibitem{drmota_tichy1997:sequences_discrepancies_and}
M.~Drmota and R.~F. Tichy, \emph{Sequences, discrepancies and applications},
  Lecture Notes in Mathematics, vol. 1651, Springer-Verlag, Berlin, 1997.

\bibitem{harman1981:trigonometric_sums_over}
G.~Harman, \emph{Trigonometric sums over primes. {I}}, Mathematika \textbf{28}
  (1981), no.~2, 249--254 (1982).

\bibitem{iwaniec_kowalski2004:analytic_number_theory}
H.~Iwaniec and E.~Kowalski, \emph{Analytic number theory}, American
  Mathematical Society Colloquium Publications, vol.~53, American Mathematical
  Society, Providence, RI, 2004.

\bibitem{katai1977:sum_digits_primes}
I.~K{\'a}tai, \emph{On the sum of digits of primes}, Acta Math. Acad. Sci.
  Hungar. \textbf{30} (1977), no.~1--2, 169--173.

\bibitem{kuipers_niederreiter1974:uniform_distribution_sequences}
L.~Kuipers and H.~Niederreiter, \emph{Uniform distribution of sequences},
  Wiley-Interscience [John Wiley \& Sons], New York, 1974, Pure and Applied
  Mathematics.

\bibitem{madritsch_thuswaldner_tichy2008:normality_numbers_generated}
M.~G. Madritsch, J.~M. Thuswaldner, and R.~F. Tichy, \emph{Normality of numbers
  generated by the values of entire functions}, J. Number Theory \textbf{128}
  (2008), no.~5, 1127--1145.

\bibitem{madritsch_tichy2013:construction_normal_numbers}
M.~G. Madritsch and R.~F. Tichy, \emph{Construction of normal numbers via
  generalized prime power sequences}, J. Integer Seq. \textbf{16} (2013),
  no.~2, Article 13.2.12, 17.

\bibitem{madritsch2012:summatory_function_q}
M.~G. Madritsch, \emph{The summatory function of $q$-additive functions on
  pseudo-polynomial sequences}, J. Th\'eor. Nombres Bordeaux \textbf{24}
  (2012), 153--171.

\bibitem{mauduit_rivat2010:sur_un_probleme}
C.~Mauduit and J.~Rivat, \emph{Sur un probl\`eme de {G}elfond: la somme des
  chiffres des nombres premiers}, Ann. of Math. (2) \textbf{171} (2010), no.~3,
  1591--1646.

\bibitem{nakai_shiokawa1990:class_normal_numbers}
Y.~Nakai and I.~Shiokawa, \emph{A class of normal numbers}, Japan. J. Math.
  (N.S.) \textbf{16} (1990), no.~1, 17--29.

\bibitem{nakai_shiokawa1992:discrepancy_estimates_class}
\bysame, \emph{Discrepancy estimates for a class of normal numbers}, Acta
  Arith. \textbf{62} (1992), no.~3, 271--284.

\bibitem{nakai_shiokawa1997:normality_numbers_generated}
\bysame, \emph{Normality of numbers generated by the values of polynomials at
  primes}, Acta Arith. \textbf{81} (1997), no.~4, 345--356.

\bibitem{schiffer1986:discrepancy_normal_numbers}
J.~Schiffer, \emph{Discrepancy of normal numbers}, Acta Arith. \textbf{47}
  (1986), no.~2, 175--186.

\bibitem{shiokawa1974:sum_digits_prime}
I.~Shiokawa, \emph{On the sum of digits of prime numbers}, Proc. Japan Acad.
  \textbf{50} (1974), 551--554.

\bibitem{tenenbaum1995:introduction_la_theorie}
G.~Tenenbaum, \emph{Introduction \`a la th\'eorie analytique et probabiliste
  des nombres}, second ed., Cours Sp\'ecialis\'es [Specialized Courses],
  vol.~1, Soci\'et\'e Math\'ematique de France, Paris, 1995.

\bibitem{vaaler1985:some_extremal_functions}
J.~D. Vaaler, \emph{Some extremal functions in {F}ourier analysis}, Bull. Amer.
  Math. Soc. (N.S.) \textbf{12} (1985), no.~2, 183--216.

\bibitem{vinogradov2004:method_trigonometrical_sums}
I.~M. Vinogradov, \emph{The method of trigonometrical sums in the theory of
  numbers}, Dover Publications Inc., Mineola, NY, 2004, Translated from the
  Russian, revised and annotated by K. F. Roth and Anne Davenport, Reprint of
  the 1954 translation.

\end{thebibliography}
\end{document}